\documentclass[12pt,a4paper]{amsart}

\usepackage{amsmath,amsthm,amssymb,latexsym,a4wide,tikz,multicol,tikz-cd}
\usepackage{arydshln,multirow}
\usepackage{tikz-qtree,tikz-qtree-compat}
\usepackage{mathtools,stmaryrd}
\usepackage{tikz}
\tikzset{font=\small}
\usepackage{enumerate}
\usetikzlibrary{matrix,arrows}
\usetikzlibrary{positioning}
\usetikzlibrary{cd}

\usepackage{mathrsfs}
\usepackage{hyperref}

\usepackage{thmtools}

\newtheorem{theorem}{Theorem} [section]
\newtheorem{lemma}[theorem]{Lemma}

\newtheorem{corollary}[theorem]{Corollary}
\newtheorem{proposition}[theorem]{Proposition}

\theoremstyle{definition}
\newtheorem{definition}[theorem]{Definition}

\newcommand{\N}{\mathbb N}
\newcommand{\F}{\mathcal{F}}

\newcommand{\NC}{\operatorname{NC}}
\newcommand{\MC}{\mathcal{M}\operatorname{C}}
\newcommand{\FS}{\operatorname{FS}}

\newcommand{\nN}{\operatorname{non}(\mathcal N)}
\newcommand{\w}{\omega}

\newcommand{\Z}{\mathcal{Z}}

\newcommand{\cl}{\operatorname{cl}}
\newcommand{\Hf}{\mathcal{H}_{\text{fin}}}

\title{Extensions of Hindman's theorem via finite colorings of topological groups}
\author{Serhii Bardyla}
\address{University of Vienna, Institute of Mathematics, Kolingasse 14-16, 1090 Vienna, Austria.}
\email{sbardyla@gmail.com}

\thanks{The research of the author was funded by the Austrian Science Fund FWF [10.55776/ESP399]}

\makeatletter


\begin{document}

\begin{abstract}
We study the partition regular properties of topological groups, proving seve\-ral extensions of Hindman’s theorem where monochromatic sets of finite sums are required to satisfy additional topological constraints. In particular, our results imply that for every nowhere dense set $C \subseteq \mathbb{R}^n$, there exists an open set $P \supseteq C$ such that for any finite coloring of $\mathbb{Q}^n \setminus P$, there is a family $\mathcal{A}$ of sequences in $\mathbb{Q}^n \setminus P$ which satisfies the following properties: (i) for each $A\in\mathcal A$, the set $\FS(A)$ of finite sums of $A$ is a closed discrete subset of $\mathbb R^n$; (ii)~the set $\bigcup_{A\in\mathcal A}\FS(A)$ is monochromatic; and (iii) the set $\bigcup_{A\in\mathcal A}\FS(A)$ is dense in an open unbounded subset of $\mathbb R^n$. 
\end{abstract}

\maketitle

\section{Introduction and main results}

In this paper, $\N$ stands for the set of positive integers. We agree to identify sequences with their images. Unless stated otherwise, sequences are assumed to be injective.
For a sequence $A=\{a(n): n\in\N\}$ in a semigroup $S$ let 
$$\FS(A)=A\cup \{a(n_1)a(n_2)\cdots a(n_k): k\geq 2 \hbox{ and } n_1<n_2<\ldots< n_k\}.$$ 
We retain the standard additive notation $\FS(A)$, even when the underlying operation is written multiplicatively, to maintain a unified language with our primary applications in $\mathbb{R}^n$ and $\mathbb{Q}^n$. The following result, originally established for the semigroup $(\mathbb{N},+)$ by Hindman~\cite{H}, is a cornerstone of Ramsey theory.

\begin{theorem}[Hindman--Galvin--Glazer]\label{Hindman}
For each finite coloring of an infinite semigroup $S$ there exists a sequence $A$ in $S$ such that the set $\FS(A)$ is monochromatic.     
\end{theorem}

Hindman's theorem has inspired numerous directions in Ramsey theory (see~~\cite{AC, B, BK, Ber, BSa, C, CZ, FNM, FK, G, HS, K, Kra, L, L1, M, S, T, Tod} and references therein). For instance, Erdős asked whether subsets of $\mathbb{N}$ with positive upper density contain Hindman-type structures. This question has seen notable recent progress in~\cite{Kra}. In this paper, we pursue a different direction by investigating finite colorings of topological groups. Specifically, we extend Hindman’s theorem to construct monochromatic sets of the form $\FS(A)$ that avoid open sets associated with a given ideal on a topological group.

We refer to a topological space simply as a space. A space $X$ is called $T_1$ if each singleton in $X$ is closed.
By $o(X)$ we denote the set of all open subsets of the space $X$. The closure of a subset $A\subseteq X$ is denoted by $\cl_X(A)$, or simply $\overline{A}$ if the ambient space is clear from the context. We use the notation $\mathcal{P}(X)$ to represent the set of all subsets of $X$.
A nonempty subset $J\subseteq \mathcal P(X)$ is called an {\em ideal on} $X$ if 
\begin{enumerate}[\rm(i)]
    \item $X\notin J$;
    \item $A\cup B\in J$ for each $A,B\in J$;
    \item if $B\in J$ and $A\subseteq B$, then $A\in J$.
\end{enumerate}
A family $\mathcal B\subseteq J$ is called a {\em base} of the ideal $J$ if for each $A\in J$ there exists $B\in\mathcal B$ such that $A\subseteq B$. Let $J^+=\mathcal P(X)\setminus J$. Elements of $J^+$ are called {\em positive with respect to} $J$.
  
Next, we provide several definitions of particular importance to this paper.
\begin{definition}
Let $J$ be an ideal on a space $X$ and $Y$ be a subset of $X$. The set $Y$ is called {\em $J$-compact in $X$}, if for each map $\phi: J\rightarrow o(X)$ such that $A\subseteq \phi(A)$ for all $A\in J$ there exists a finite subset $F\subseteq J$ with $Y\subseteq \bigcup_{x\in F}\phi(x)$. The space $X$ is called {\em $J$-compact} if $X$ is $J$-compact in itself.
\end{definition}



\begin{definition}
An ideal $J$ on a space $X$ is called 
\begin{itemize}
\item {\em exhaustive} if there exists a family $\{A_i: i\in \N\}\subseteq J\cap o(X)$ such that $\bigcup_{i\in\N}A_i=X$ and $A_i\subseteq A_{i+1}$ for all $i\in\N$;
\item {\em near-maximal} if there exists an ultrafilter $u$ on $X$ such that $\{X\setminus \overline{U}: U\in u\}\subseteq J$.
\end{itemize}

\end{definition}

\begin{definition}\label{defb}
 Let $J$ be an ideal on a space $X$. If there exists a minimal cardinal $\kappa$ which satisfies the following condition: 

 \begin{itemize}
\item[($*$)] there exists $A\in J$ such that for each $B\in J$ with $A\subseteq B$ there exists a family $\{U_\alpha: \alpha\in \kappa\}$ consisting of open sets which contain $B$ such that for every open set $W\supseteq B$ there exists $\alpha\in\kappa$ with $W\setminus U_\alpha\notin J$,
 \end{itemize}
 then put $\chi(J)=\kappa$. If such a cardinal $\kappa$ doesn't exist, then set $\chi(J)=|X|^+$.   
\end{definition}

Notice that if an ideal $J$ on a space $X$ has a base consisting of open sets, then there is no cardinal which satisfies condition ($*$) from Definition~\ref{defb} with respect to $J$.  In Proposition~\ref{card} we show that $\chi(J)\geq \aleph_1$ for every exhaustive ideal on an infinite space $X$.

For an element $s$ of a semigroup $S$ the {\em left shift} $\lambda_s: S\rightarrow S$ is defined by $\lambda_s(x)=sx$. Dually, for every $s\in S$ the {\em right shift} $\rho_s: S\rightarrow S$ is defined by $\rho_s(x)=xs$. 
A semigroup $S$ endowed with a topology is called {\em left topological} (resp., {\em right topological}) if for each $s\in S$ the shift $\lambda_s$ (resp., $\rho_s$) is continuous. 

\begin{definition}
 An ideal $J$ on a semigroup $S$ is called 
 {\em left-invariant} if for every $A\in J$ and $x\in S$ we have $xA\in J$.
\end{definition}

A subset $A$ of a semigroup $S$ is called a {\em left ideal} if $SA\subseteq A$. We hope that no confusion will arise from the similarity in terminology between the algebraic notion of a left ideal and the set theoretic notion of an ideal.

\begin{definition}
Let $S$ be a subsemigroup of a left topological group $G$. An ideal $J$ on $G$ is called {\em nicely exhaustive with respect to} $S$, if there exists a family $\{A_i: i\in \N\}\subseteq J$ consisting of open sets such that 
\begin{enumerate}[\rm(i)]
    \item $\bigcup_{i\in\N}A_i=G$ and $A_i\subseteq A_{i+1}$ for all $i\in\N$;
    \item $S\setminus A_i$ is a nonempty left ideal in $S$ for each $i\in\N$.
\end{enumerate}   
\end{definition}

The following result is the main general result of this paper.

\begin{restatable}{theorem}{maintheorem}\label{main}
   Let $J$ be a left-invariant exhaustive ideal on a $T_1$ left topological group $G$, $C$ be a fixed element of $J$, and $S$ be a subsemigroup of $G$ such that $S$ is not $J$-compact in $G$ and $|S|<\chi(J)$. Then there exists an open subset $P$ of $G$ such that $C\subseteq P$, and for every finite coloring of $S\setminus P$ there exists a family $\mathcal A$ of  sequences in $S\setminus P$ which satisfies the following conditions:
   \begin{enumerate}[\rm(i)]
      \item each $A\in\mathcal A$ is closed and discrete in $G$;
      \item the set $\bigcup_{A\in\mathcal A}\FS(A)$ is monochromatic;
      \item $\FS(A)\cap \FS(B)=\emptyset$ for all distinct $A,B\in\mathcal A$;
      \item $\bigcup_{A\in\mathcal A}\FS(A)\in J^+$.
   \end{enumerate}
   Moreover,  
   \begin{enumerate}[\rm(i)]
      \setcounter{enumi}{4}
      \item if $J$ is additionally nicely exhaustive with respect to $S$, then $\FS(A)$ is a closed discrete subset of $G$ for all $A\in\mathcal A$;
      \item if $J$ is not a near-maximal ideal, then $P\in J^+$.
   \end{enumerate}
\end{restatable}



\begin{definition}
Let $\NC(X)$ be the ideal on a non-compact Hausdorff space $X$ which consists of sets that are contained in the union of a compact subset of $X$ and a nowhere dense subset of $X$.   
\end{definition}

A separable completely metrizable topological space is called {\em Polish}. 
Another principal result of this paper is the following indirect consequence of Theorem~\ref{main}.

\begin{restatable}{theorem}{Polishappl}\label{appl}
 Let $G$ be a locally compact Polish left topological group, $C\in \NC(G)$, and $S$ be a subsemigroup of $G$ such that $S$ is dense in an open subset of $G$ with non-compact closure.  Then there exists an open subset $P$ of $G$ such that $C\subseteq P$, $P\in \NC(G)^+$, and for every finite coloring of $S\setminus P$ there exists a family $\mathcal A$ of sequences in $S\setminus P$ which satisfies the following conditions: 
   \begin{enumerate}[\rm(i)]
  \item  each $A\in\mathcal A$ is closed and discrete in $G$;
  \item the set $\bigcup_{A\in\mathcal A}\FS(A)$ is monochromatic;
  \item $\FS(A)\cap \FS(B)=\emptyset$ for all distinct $A,B\in\mathcal A$;
  \item the set $\bigcup_{A\in\mathcal A}\FS(A)$ is dense in an open subset of $G$ with non-compact closure.
  \end{enumerate}  
  Moreover, if $\NC(G)$ is additionally nicely exhaustive with respect to $S$, then
  \begin{enumerate}[\rm(v)]
  \item $\FS(A)$ is a closed discrete subset of $G$ for all $A\in\mathcal A$.
  \end{enumerate}  
\end{restatable}

In this paper, we consider $\mathbb{R}^n$ as an additive group with its standard Euclidean topology.
For each $m\in\N$ let 
$B_m=\{(x_1,\ldots, x_n)\in\mathbb R^n: |x_i|<m \hbox{ for all }i\leq n\}$.
Note that the family $\{B_m:m\in\N\}\subseteq \NC(\mathbb R^n)$ witnesses that the ideal $\NC(\mathbb R^n)$ is nicely exhaustive with respect to $\mathbb Q^n_+=\{(x_1,\ldots, x_n)\in\mathbb Q^n: x_i>0 \hbox{ for all }i\leq n\}$. Hence Theorem~\ref{appl} applied to $S=\mathbb Q^n_+$ and $T=\mathbb R^n$ implies the following.  

\begin{corollary}\label{NC}
For every $C\in \NC(\mathbb R^n)$ there exists an open set $P\in\NC(\mathbb R^n)^+$ such that $C\subseteq P$, and for every finite coloring of $\mathbb Q^n_+\setminus P$ there exists an infinite family $\mathcal A$ of sequences in $\mathbb Q^n_+\setminus P$ which satisfies the following conditions: 
   \begin{enumerate}[\rm(i)]
    \item $\FS(A)$ is a closed discrete subset of $\mathbb R^n$ for all $A\in\mathcal A$;
      \item the set $\bigcup_{A\in\mathcal A}\FS(A)$ is monochromatic; 
     \item $\FS(A)\cap \FS(B)=\emptyset$ for all distinct $A,B\in\mathcal A$;
  \item the set $\bigcup_{A\in\mathcal A}\FS(A)$ is dense in an open unbounded subset of $\mathbb R^n$.
  \end{enumerate}  
\end{corollary}

This paper is organized as follows. 
In Section~\ref{2}, we introduce $J$-remote ultrafilters and present the proof of Theorem~\ref{main}. In Section~\ref{3}, using remote points in Stone--\v{C}ech compactifications, we prove Theorem~\ref{appl}. In Section~\ref{4}, we discuss further applications of Theorem~\ref{main} to other ideals on topological groups, including the meager ideal, the ideal of sets of finite Haar measure, and the ideal of sets of Haar asymptotic density zero.

\section{$J$-remote ultrafilters and the proof of Theorem~\ref{main}}\label{2}


A subset $\F\subseteq \mathcal P(X)$ is called a {\em filter on} $X$ if the family $\{X\setminus U: U\in u\}$ is an ideal on $X$. In other words, $\F$ is a filter on $X$ if $X\in\F$, $\emptyset\notin \F$, and $\F$ is closed under finite intersections and taking supersets. A filter $\F$ on $X$ is called an {\em ultrafilter} if $\F$ is maximal with respect to the inclusion among all filters on $X$ or, equivalently, for each $A\subseteq X$ either $A\in\F$ or $X\setminus A\in \F$. A family $\mathcal B\subseteq \F$ is called a {\em base} of a filter $\F$ if for each $F\in\F$ there exists $B\in\mathcal B$ such that $B\subseteq F$.

The Stone-\v{C}ech compactification $\beta (X)$ of a discrete space $X$ is the set of all ultrafilters on $X$ endowed with a topology $\tau$ given by the base $\mathcal B=\{\langle A\rangle: A\subseteq X\}$, where $$\langle A\rangle=\{u\in\beta (X): A\in u\}.$$ Recall that each element $x\in X$ is identified with the principal ultrafilter $\{A\subseteq X: x\in A\}$. If $S$ is a discrete semigroup, then the semigroup operation on $S$ can be canonically lifted to a semigroup operation on $\beta (S)$ as follows: if $u,v\in\beta (S)$, then $uv$ is the ultrafilter generated by the base consisting of the sets $\bigcup_{x\in U}xV_x$, where $U\in u$ and $\{V_x:x\in U\}\subseteq v$ are arbitrary. Equivalently, $$A\in uv \iff \{
s\in S : \lambda_s^{-1}(A)\in v\} \in u.$$  
By~\cite[Theorem~4.1]{HS}, the defined above semigroup operation on $\beta (S)$ is unique among those extending the operation of $S$ and satisfying the following two natural conditions:
\begin{itemize}
    \item[(i)] for each $u\in\beta (S)$ the right shift $\rho_u$ is continuous;
    \item[(ii)] for each $s\in S$ the left shift $\lambda_s$ is continuous.
\end{itemize}
Thus, for any discrete semigroup $S$, $\beta (S)$ endowed with the aforementioned operation is a compact right topological semigroup.
For more about the algebraic structure of $\beta (S)$ see the monograph~\cite{HS} and references therein.  

For a space $X$ by $X_d$ we denote the set $X$ endowed with the discrete topology. 

\begin{definition}
Let $J$ be an ideal on a space $X$ and $Y$ be a subset of $X$. Put 
$$\mathcal R_J(Y)=\{u\in\beta (Y_d): \forall A\in J \hbox{ } \exists U\in u \hbox{ such that }\cl_X(U)\cap A=\emptyset\}.$$
Elements of $\mathcal R_J(Y)$ are called {\em $J$-remote ultrafilters on} $Y$. 
\end{definition}




\begin{lemma}\label{RJ}
Let $Y$ be a subset of a space $X$.
Then $\mathcal R_J(Y)\neq \emptyset$ if and only if $Y$ is not $J$-compact in $X$.    
\end{lemma}

\begin{proof}
    Assuming that $\mathcal R_J(Y)\neq \emptyset$, fix $u\in \mathcal R_J(Y)$. Then for each $A\in J$ there exists $U_A\in u$ such that $\cl_X(U_A)\cap A=\emptyset$. Define a function $\phi: J\rightarrow o(X)$ by $\phi(A)=X\setminus \cl_X(U_A)$. Then for each finite family $\mathcal A\subseteq J$ the set $Y\setminus \bigcup_{A\in\mathcal A}\phi(A)$ is nonempty, as it contains the set $\bigcap_{A\in\mathcal A}U_A$ which belongs to $u$. Hence $Y$ is not $J$-compact in $X$. 

    Assume that $Y$ is not $J$-compact in $X$. Then there exists a function $\phi: J\rightarrow o(X)$ such that $A\subseteq \phi(A)$ for all $A\in J$, and for each finite subset $\mathcal A\subseteq J$ the set $Y\setminus \bigcup_{A\in\mathcal A}\phi(A)$ is nonempty. Observe that the family 
     $$\mathcal B=\{Y\setminus \bigcup_{A\in\mathcal A}\phi(A): \mathcal A\hbox{ is a finite subset of } J\}$$
     is closed under finite intersections. It is clear that every ultrafilter $u$ on $Y$ that contains the family $\mathcal B$ belongs to $\mathcal R_J(Y)$.    
\end{proof}

An ideal $J$ on a semigroup $S$ is called  {\em preimage-stable} if for every $A\in J$ and $x\in S$ we have $\lambda_x^{-1}(A)\in J$. Noteworthy, an ideal $J$ on a group $G$ is preimage-stable if and only if $J$ is left-invariant. A map $f: X\rightarrow Y$ is called {\em closed} if $f(A)$ is closed in $Y$ for each closed subset $A$ of $X$.

\begin{lemma}\label{stable}
Let $J$ be a preimage-stable ideal on a semigroup $T$ equipped with a topology such that the map $\lambda_x$ is closed for each $x\in T$, and $S$ be a subsemigroup of $T$ which is not $J$-compact in $T$. Then for each $s\in S$ and $u\in \mathcal R_J(S)$ we have $su\in  \mathcal R_J(S)$.    
\end{lemma}

\begin{proof}
Lemma~\ref{RJ} yields that $\mathcal R_J(S)\neq \emptyset$.
Fix any $A\in J$, $s\in S$ and $u\in \mathcal R_J(S)$.
In order to show that $su\in \mathcal R_J(S)$, we are going to find an element $B\in su$ such that $\cl_T(B)\cap A=\emptyset$. Since the ideal $J$ is preimage-stable, we get $\lambda_s^{-1}(A)\in J$. Since $u\in\mathcal R_J(S)$, there exists $U\in u$ such that $\cl_T(U)\cap \lambda_s^{-1}(A)=\emptyset$. It is clear that $s\cl_T(U)\cap A=\emptyset$. Since $U\subseteq \cl_T(U)\cap S$, we get that $B=s(\cl_T(U)\cap S)\in su$. 
The set $s\cl_T(U)$ is closed in $T$, as $\lambda_s$ is a closed map. Thus $$\cl_T(B)\cap A\subseteq \cl_T(s\cl_T(U))\cap A=s\cl_T(U)\cap A=\emptyset,$$
as required
\end{proof}

\begin{proposition}\label{key}
Let $J$ be a preimage-stable ideal on a semigroup $T$ equipped with a topology such that $\lambda_x$ is a closed map for each $x\in T$. Let $S$ be a subsemigroup of $T$ which is not $J$-compact in $T$, and $u$ be an ultrafilter on $S$ which contains an element of cardinality $\kappa<\chi(J)$. Then $uv\in \mathcal R_J(S)$ for every $v\in \mathcal R_J(S)$.
\end{proposition}

\begin{proof}
Fix any $A\in J$. Since $\kappa<\chi(J)$, 
there exists $A'\in J$ such that $A\subseteq A'$ and for each family $\{U_\alpha: \alpha\in \kappa\}$ consisting of open sets which contain $A'$, there exists an open set $W\supseteq A'$ such that $W\setminus U_\alpha\in J$  for all $\alpha\in\kappa$. 

By the assumption, there exists $U\in u$ such that $|U|=\kappa$. Fix an enumeration $U=\{a_\xi:\xi\in \kappa\}$. In order to show that $uv\in \mathcal R_J(S)$ it suffices to find $\{V_\xi:\xi\in \kappa\}\subseteq v$ such that $\cl_T(\bigcup_{\xi\in \kappa}a_\xi V_\xi)\cap A'=\emptyset$. By Lemma~\ref{stable}, for each $\xi\in \kappa$ the ultrafilter $a_\xi v$ is $J$-remote. It follows that for every $\xi\in \kappa$ there exists $B_\xi\in v$ such that $\cl_T(a_\xi B_\xi)\cap A'=\emptyset$. 
Note that for each $\xi\in\kappa$, the open set $C_\xi=T\setminus \cl_T(a_\xi B_\xi)$ contains $A'$. By the choice of $A'$, there exists an open set $W\supseteq A'$ such that $W\setminus C_\xi\in J$ for all $\xi\in\kappa$. Notice that for every $\xi\in\kappa$ we have $C_\xi\cap S\notin a_\xi v$. Since $W\setminus C_\xi\in J$ and $a_\xi v\in \mathcal R_J(S)$, we get that $(W\setminus C_\xi)\cap S\notin a_\xi v$ for every $\xi\in\kappa$. It follows that $((W\setminus C_\xi)\cap S)\cup (C_\xi\cap S)\notin a_\xi v$ for all $\xi\in\kappa$.  Since 
$$W\cap S\subseteq ((W\setminus C_\xi)\cap S)\cup (C_\xi\cap S),$$
we get $W\cap S\notin a_\xi v$ for all $\xi\in\kappa$. Then for each $\xi\in\kappa$ there exists $V_\xi\in v$ such that $a_\xi V_\xi\subseteq S\setminus W$. Since $W$ is an open subset of $T$ that contains $A'$, we get  
$\cl_T(\bigcup_{\xi\in \kappa}a_\xi V_\xi)\cap A'=\emptyset$, as required.
\end{proof}

It is well-known that on every infinite semigroup $S$ there exists an idempotent ultrafilter $u$, i.e., $uu=u$. 
An idempotent ultrafilter $u$ on $S$ is called {\em minimal} if for every idempotent ultrafilter $v$ on $S$, either of the equalities $vu=v$ or $uv=v$ implies $u=v$.

\begin{corollary}\label{useful}
Let $J$ be a preimage-stable ideal on a semigroup $T$ equipped with a topology such that $\lambda_x$ is a closed map for each $x\in T$. Let $S$ be a subsemigroup of $T$ such that $S$ is not $J$-compact in $T$ and $|S|<\chi(J)$. Then $\mathcal R_J(S)$ is a left ideal in $\beta(S_d)$. In particular, there exists a $J$-remote minimal idempotent ultrafilter on $S$.     
\end{corollary}

\begin{proof}
Since $S$ is not $J$-compact in $G$, Lemma~\ref{RJ} implies $\mathcal R_J(S)\neq \emptyset$. Since $|S|<\chi(J)$, Proposition~\ref{key} implies that $\mathcal R_J(S)$ is a left ideal in $\beta(S_d)$. By \cite[Corollary 2.6]{HS}, $\mathcal R_J(S)$ contains a minimal (in $\beta(S_d)$) ideal. Theorems 2.8 and 2.9 from \cite{HS} imply that $\mathcal R_J(S)$ contains a minimal idempotent. 
\end{proof}

The following result establishes a lower bound for the cardinal $\chi(J)$, assuming the ideal $J$ is exhaustive.

\begin{proposition}\label{card}
Let $J$ be an exhaustive ideal on an infinite space $X$. Then $\chi(J)\geq \aleph_1$.    
\end{proposition}

\begin{proof}
If $\chi(J)=|X|^+$, then obviously $\chi(J)\geq \aleph_1$, as $X$ is infinite.

Assume that there exists a minimal cardinal $\kappa$ which satisfies condition ($*$) from Definition~\ref{defb}. We are going to show that $\kappa\geq \aleph_1$.
Since the ideal $J$ is exhaustive, there exists an ascending family $\{A_n:n\in\N\}\subseteq J\cap o(X)$ such that $\bigcup_{n\in\N}A_n=X$. Consider any $B\in J$ and a countable family $\{U_n: n\in\N\}\subseteq o(X)$ such that $B\subseteq U_n$ for all $n\in\N$. 
For every $n\in\N$ let $V_n=\bigcap_{i\leq n}U_i\cap A_n$. It is clear that the set $V_n$ is open and  contains $B\cap A_n$. Put $V=\bigcup_{n\in\N}V_n$. It is clear that $V$ is an open set that contains $B$. Moreover, for every $n \in\N$ we have 
$$V\setminus U_n\subseteq \bigcup_{i<n}V_i\subseteq \bigcup_{i<n}A_i\in J.$$ 
Hence $\kappa=\chi(J)\geq \aleph_1$.  
\end{proof}

To keep this paper self-contained we provide a proof of the following folklore result.

\begin{lemma}[Folklore]\label{folk}
Let $G$ be a first-countable topological group, $\{U_n:n\in\N\}$ be an open neighborhood base at $1_G$, and $A$ be a compact subset of $G$. Then for each open set $W\supseteq A$ there exists $n\in\N$  such that $AU_n\subseteq W$. In particular, the countable family $\{AU_n: n\in\N\}$ is an outer open base at the set $A$. 
\end{lemma}

\begin{proof}
Since the set $W\supseteq A$ is open and $G$ is a topological group, for each $a\in A$ there exists $U_{n(a)}$ such that $aU_{n(a)}U_{n(a)}\subseteq W$. Put $V_a=aU_{n(a)}$ for every $a\in A$. Since the set $A$ is compact and $\{V_a:a\in A\}$ is an open cover of $A$, there exists a finite subset $F$ of $A$ such that $A\subseteq \bigcup_{f\in F}V_f$. Let $m\in\N$ be such that $U_m\subseteq\bigcap_{f\in F}U_{n(f)}$. In order to show that $AU_m\subseteq W$, fix any $a\in A$ and $y\in U_m$. There exists $f\in F$ such that $a\in V_f$. It follows that $a=fx$ for some $x\in U_{n(f)}$. Then $$ay=fxy\in fU_{n(f)}U_m\subseteq fU_{n(f)}U_{n(f)}\subseteq W.$$  Hence $AU_m\subseteq W$, as required.   
\end{proof}

A family $\mathcal A\subseteq \N^\N$ is called {\em unbounded} if for each $f\in \N^\N$ there exists $g\in\mathcal A$ such that $|\{n\in\N: f(n)<g(n)\}|=\aleph_0$. The minimal cardinality of an unbounded family in $\N^\N$ is denoted by $\mathfrak b$. It is known that $\aleph_1\leq \mathfrak b\leq 2^{\aleph_0}$, see~\cite{Blass}. The following result establishes another lower bound for the cardinal $\chi(J)$ under certain natural restrictions. 

\begin{proposition}\label{b}
Let $J$ be an ideal on a locally compact non-discrete Polish topological group $G$ with the following properties:
\begin{enumerate}[\rm(i)]
\item $J$ has a base consisting of $\sigma$-compact sets;
\item there exists a family $\{V_n:n\in\N\}\subseteq J\cap o(G)$ such that $\bigcup_{n\in\N}V_n=G$, and for each compact subset $K$ of $G$ there exists $n\in\N$ satisfying $K\subseteq V_n$.
\end{enumerate}
Then $\chi(J)\geq \mathfrak b$.    
\end{proposition}

\begin{proof}
Since $G$ is a non-discrete Polish topological group, we have $|G|= 2^{\aleph_0}$. Thus if there exists no cardinal $\kappa$ which satisfies condition ($*$) from Definition~\ref{defb} with respect to the ideal $J$, then $\chi(J)=|G|^+> 2^{\aleph_0}\geq \mathfrak b$. 

Assume that there exists a minimal cardinal $\kappa$ which satisfies condition ($*$) from Definition~\ref{defb} with respect to $J$. That is there exists $A\in J$ such that for each $C\in J$ with $A\subseteq C$ there exists a family $\mathcal Q_C$ of size $\kappa$ consisting of open sets which contain $C$ such that for every open set $W\supseteq C$ there exists $Q\in\mathcal Q_C$ with $W\setminus Q\notin J$. Since $J$ has a base consisting of $\sigma$-compact sets, there exists a $\sigma$-compact set $B\in J$ which contains $A$. Let $B=\bigcup_{i\in\N}B_i$ where for each $i\in\N$ the set $B_i$ is compact. 
Also, for every $i\in\N$ let $t_i=\min\{n\in\N: B_i\subseteq V_n\}$, which exists by the assumption. 
Lemma~\ref{folk} implies that for each $i\in\N$ we can fix a family $\{U_n^i:n\in\N\}\subseteq o(G)$ such that $B_i\subseteq U_{n+1}^i\subseteq U_n^i$ for all $n\in\N$ and for each open $O\supseteq B_i$ there exists a minimal $n(O,i)\in\N$ such that $U_{n(O,i)}^i\subseteq O$. Substituting $U_n^i$ with $U_n^i\cap V_{t_i}$ if necessary, we additionally assume that $U_n^i\subseteq V_{t_i}$ for all $i,n\in\N$.

To derive a contradiction, assume that $\kappa<\mathfrak b$. Each $Q\in\mathcal Q_B$ defines a function $f_Q\in\N^\N$ as follows: $f_Q(i)=n(Q,i)$. Since $|\mathcal Q_B|<\mathfrak b$, there exists a function $g\in\N^\N$ such that the set $\{i\in\N :f_Q(i)>g(i)\}$ is finite for every $Q\in\mathcal Q_B$. Consider the set $W=\bigcup_{i\in\N}U^i_{g(i)}$. Since $B_i\subseteq U^i_{g(i)}$ for all $i\in\N$, we get  $B\subseteq W$. It is clear that $W$ is open. Fix any $Q\in\mathcal Q_B$. By the choice of $g$, there exists $k\in\N$ such that for each $i\geq k$ we have $g(i)\geq f_Q(i)=n(Q,i)$. It follows that $U^i_{g(i)}\subseteq U^{i}_{n(Q,i)}\subseteq Q$ for each $i\geq k$.
  Hence $$W\setminus Q\subseteq \bigcup_{i<k}U^i_{g(i)}\subseteq \bigcup_{i<k}V_{t_i}\in J.$$
 It follows that $W\setminus Q\in J$ for all $Q\in \mathcal Q_B$, which contradicts the choice of $\mathcal Q_B$. The obtained contradiction yields $\chi(J)\geq \mathfrak b$.
\end{proof}

Recall that a family $\mathcal A$ of subsets of a space $X$ is called {\em locally finite} if each $x\in X$ has an open neighborhood which intersects only finitely many elements of $\mathcal A$.

\medskip

We are in a position to prove Theorem~\ref{main}, which we restate for the reader's convenience. 

\maintheorem*

\begin{proof}
By Corollary~\ref{useful}, there exists a $J$-remote idempotent ultrafilter $u$ on $S$. It follows that there exists $U\in u$ such that $\cl_G(U)\cap C=\emptyset$. Put $P=G\setminus \cl_G(U)$. It is clear that $P$ is an open set which contains $C$. Fix a finite coloring $c: S\setminus P\rightarrow \{1,\ldots, n\}$. Since $S\setminus P\in u$ and $u$ is an ultrafilter, there exists $m\leq n$ such that $c^{-1}(m)\in u$. Then $D=U\cap c^{-1}(m)\in u$.  Fix an ascending family $\{O_i:i\in\N\}\subseteq J$ of open subsets of $G$ such that $\bigcup_{i\in\N}O_i=G$, which exists as the ideal $J$ is exhaustive. Since $u$ is an idempotent ultrafilter, there exist $x(1)\in D$ and  $U_{x(1)}\in u$ such that $x(1)U_{x(1)}\subseteq D$. Substituting $U_{x(1)}$ with $U_{x(1)}\cap D$, we can assume that $U_{x(1)}\subseteq D$.  
 Since the ultrafilter $u$ is $J$-remote, 
 $U_{x(1)}\setminus O_{1}\in u$. Since $u$ is an idempotent ultrafilter, there exist $$x(2)\in U_{x(1)}\setminus O_{1} \quad\hbox{ and }\quad U_{x(2)}\in u$$ such that 
 $$U_{x(2)}\subseteq U_{x(1)}\setminus O_{1}\quad \hbox{ and } \quad x(2)U_{x(2)}\subseteq U_{x(1)}\setminus O_{1}.$$  
 Since $u$ is a $J$-remote idempotent ultrafilter there exist $$x(3)\in U_{x(2)}\setminus O_{2}\quad \hbox{ and }\quad U_{x(3)}\in u$$ such that $$U_{x(3)}\subseteq U_{x(2)}\setminus O_2 \quad \hbox{ and } \quad x(3)U_{x(3)}\subseteq U_{x(2)}\setminus O_{2}.$$ 
 Proceeding this way we construct a decreasing sequence $\{U_{x(n)}:n\in\N\}\subseteq u$ and a sequence $\{x(n): n\in\N\}\subseteq D$ such that $$x(n)\in U_{x(n-1)}\setminus O_{n-1}, \quad U_{x(n)}\subseteq U_{x(n-1)}\setminus O_{n-1}\quad \hbox{ and }\quad x(n)U_{x(n)}\subseteq U_{x(n-1)}\setminus O_{n-1}$$for each $n\in\N$, where we agree to assume that $U_{x(0)}=D$ and $O_0=\emptyset$.
 
Put $A_0=\{x(n): n\in\N\}$.  
Then for $m>n$ we have $$x(n)x(m)\in x(n)U_{x(m-1)}\subseteq x(n)U_{x(n)}\subseteq U_{x(n-1)}.$$
Keeping in mind the formula above, it is routine to check that for each finite increasing  sequence of natural numbers $\{n_i: i\leq k\}$ we have $x(n_1)x(n_2)\cdots x(n_k)\subseteq U_{x(n_1-1)}$.
It follows that $\FS(A_0)\subseteq D$. By the construction, for each $n\in\N$ we have $O_n\cap A_0\subseteq \{x(0),\ldots, x(n)\}$, witnessing that the family $\{\{x(n)\}: n\in\N\}$ is locally finite in $G$. Since each singleton is closed in $G$, \cite[Theorem 1.1.11]{Eng} implies 
$$\cl_G(A_0)=\cl_G(\bigcup_{n\in\N}\{x(n)\})=\bigcup_{n\in\N}\cl_G(\{x(n)\})=\bigcup_{n\in\N}\{x(n)\}=A_0.$$
Hence $A_0$ is a closed subset of $G$. Since the space $G$ is $T_1$, the aforementioned inclusion $O_n\cap A_0\subseteq \{x(0),\ldots, x(n)\}$ implies that the set $A_0$ is discrete. 

Next we are going to construct a family $\mathcal A$ of sequences in $S\setminus P$ which satisfies conditions (i) - (iv).
If $\FS(A_0)\notin J$, then the family $\mathcal A=\{A_0\}$ satisfies conditions (i)~-~(iv). Otherwise, assume that for some $\xi\in|D|^+$ we have already constructed a family $\mathcal A_\xi=\{A_\alpha: \alpha\in\xi\}$ consisting of closed discrete sequences in $D$ such that $\bigcup_{A\in\mathcal A_\xi}\FS(A)\subseteq D$ and $\FS(A)\cap \FS(B)=\emptyset$ for all distinct $A,B\in\mathcal A_\xi$. If $\bigcup_{\alpha\in\xi} \FS(A_\alpha)\in u$, then $\bigcup_{\alpha\in\xi} \FS(A_\alpha)\notin J$ and thus $\mathcal A_\xi$ satisfies conditions (i) - (iv). Otherwise, as the ultrafilter $u$ is $J$-remote, we have $D'=D\setminus \bigcup_{\alpha\in\xi} \FS(A_\alpha)\in u$ and repeating the above arguments we can find a closed discrete sequence $A_\xi\subseteq D'$ such that $\FS(A_\xi)\subseteq D'$. Note that $\mathcal A_{\xi+1}=\{A_\alpha: \alpha\in\xi+1\}$ satisfies conditions (i) - (iv) whenever $\bigcup_{\alpha\in\xi+1} \FS(A_\alpha)\in u$. It is clear that the recursion will reach an ordinal $\mu\in|D|^+$ such that $\bigcup_{\alpha\in\mu} \FS(A_\alpha)\in u$. Then the family $\mathcal A_\mu$ satisfies conditions (i) - (iv). 

Assume that $J$ is nicely exhaustive with respect to $S$. In this case we can assume that the family $\{O_i:i\in\N\}\subseteq J$ fixed at the beginning of the proof satisfies the following additional property: $S(S\setminus O_i)\subseteq S\setminus O_i$ for each $i\in\N$. Under this assumption, the constructed above sequence $A_0=\{x(n): n\in\N\}$ has the following additional property: for each finite increasing sequence $i_1 < i_2 <\ldots< i_n$ of positive integers the product $x(i_1)x(i_2)\cdots x(i_n)$ belongs to $S\setminus O_{i_n-1}$, as $x(i_n)\in S\setminus O_{i_n-1}$. It follows that for each $i\in\N$ we have $$\FS(A_0)\cap O_i\subseteq \{x(k):k\leq i\}\cup\{x(n_1)x(n_2)\cdots x(n_k): n_1<n_2<\ldots <n_k\leq i\}.$$ Since the latter set is finite, we get that the family $\{\{x\}: x\in\FS(A_0)\}$ is locally finite in $G$. Similarly as before it can be checked that $\FS(A_0)$ is a closed discrete subset of $G$. 
From this point on, we can proceed as above to construct a family $\mathcal A_\mu$ which satisfies conditions (i) - (v).

In order to establish item (vi), assume that $J$ is not a near-maximal ideal. Then there exists $V\in u$ such that $G\setminus \cl_G(V)\in J^+$. Recall that $U$ is an element of $u$ such that $\cl_G(U)\cap C=\emptyset$. Put $W=V\cap U$. It is easy to check that $P=G\setminus \cl_G(W)\in J^+$ is an open set which contains $C$, and $S\setminus P\in u$. From this point on, we can repeat the above proof to find a family $\mathcal A$ of sequences in $S\setminus P$ which satisfies conditions (i) - (iv).
\end{proof}

Condition (vi) in Theorem~\ref{main} excludes the case where $J$ is a near-maximal ideal. We suspect that no left-invariant exhaustive ideal on a $T_1$ left topological group can be near-maximal, though this remains an open question.

\section{Remote points and the proof of Theorem~\ref{appl}}\label{3}


For a Tychonoff space $X$ by $\beta(X)$ we denote the Stone-\v{C}ech compactification of $X$ (see~\cite[Chapter 3.6]{Eng}). A point $p\in\beta (X)\setminus X$ is called {\em remote}, if $p$ is not in the closure of any nowhere dense subset of $X$. We slightly abuse notation by saying that a space $X$ {\em has a remote point} if $X$ is Tychonoff and $\beta (X)\setminus X$ contains a remote point. 
The following result is a consequence of~\cite[Theorem~2.4]{D84}.  

\begin{theorem}[Dow]\label{dow}
    Each non-compact separable metrizable space has a remote point.
\end{theorem}

Recall that a subset $A$ of a space $X$ is called a {\em zero set} if $A=f^{-1}(0)$ for some continuous function $f: X\rightarrow [0,1]$. A subset $A\subseteq X$ is called {\em cozero} if $X\setminus A$ is a zero set.
The existence of remote points in $\beta(X)\setminus X$ can be characterized using the existence of certain filters on the space $X$, defined below.

\begin{definition}\rm
Let $R$ be one of the following four properties of subsets of a space $X$: open, closed, zero, cozero.
 A filter $\F$ on $X$ is called 
 \begin{enumerate}[\rm(i)]
\item an {\em $R$ filter} if $\F$ possesses a base consisting of sets with property $R$;
\item an {\em $R$ ultrafilter} if $\F$ is an $R$ filter, and for any subset $A\subseteq X$ with property $R$ either $A\in \F$ or $X\setminus A\in \F$;
\item a {\em closed-open-zero-cozero ultrafilter} if $\F$ is simultaneously a closed ultrafilter, an open ultrafilter, a zero ultrafilter, and a cozero ultrafilter.
\end{enumerate}
\end{definition}

The following result was proven in~\cite[Theorem 1.6]{BS}.

\begin{theorem}[Bardyla, \v{S}upina]\label{old2}
For a Tychonoff space $X$ the following conditions are equivalent:
\begin{enumerate}[\rm(a)]
    \item $X$ has a remote point;
    \item There exists a closed-open-zero-cozero ultrafilter on $X$ with no accumulation points.
\end{enumerate}
\end{theorem}

The following technical lemma is useful for detecting subsets of a space $X$ that are not $\NC(X)$-compact in $X$.

\begin{lemma}\label{dense}
 Let $Q$ be an open subset of a space $Y$ such that the subspace $\cl_Y(Q)$ has a remote point. Then each dense subset of $Q$ is not $\NC(Y)$-compact in~$Y$.
\end{lemma}

\begin{proof}
 Fix a dense subset $D$ of $Q$ and let $X=\cl_Y(Q)=\cl_Y(D)$ be equipped with the subspace topology inherited from $Y$. First we are going to show that $D$ is not $\NC(X)$-compact in $X$. Since $X$ has a remote point, Theorem~\ref{old2} yields a closed-open-zero-cozero ultrafilter $\F$ on $X$ with no accumulation points. Fix any nowhere dense subset $A$ of $X$. Since $\F$ is an open ultrafilter, the open dense subset $X\setminus \cl_X(A)$ of $X$ belongs to $\F$. Since $\F$ is a closed filter, there exists a closed set $F_A\in \F$ such that $F_A\subseteq X\setminus \cl_X(A)$. It follows that for each nowhere dense set $A$ of $X$ there exists an open subset $V_A=X\setminus F_A$ such that $A\subseteq V_A$ and $X\setminus V_A\in\F$. As $\F$ is a closed ultrafilter and the space $X$ is Tychonoff, for a given compact subset $K$ of $X$ either $K\in\F$ or there exists a closed set $F_K\in\F$ such that $F_K\cap K=\emptyset$. Since $\F$ does not have accumulation points, the latter case holds true. For each compact $K\subseteq X$ put $U_K=X\setminus F_K$. By the definition of $\NC(X)$, for each $B\in\NC(X)$ there exist a nowhere dense set $A(B)\subseteq X$ and a compact set $K(B)\subseteq X$ such that $B\subseteq A(B)\cup K(B)$.  Consider the map $\phi: \NC(X)\rightarrow o(X)$ defined by $\phi(B)=V_{A(B)}\cup U_{K(B)}$. By the definition of $\phi$, for each finite subset $\mathcal A\subseteq \NC(X)$ we have that $X\setminus \bigcup_{A\in\mathcal A}\phi(A)\in \F$. Since the filter $\F$ is open, there exists an open set $F_{\mathcal A}\in\F$ such that $F_\mathcal{A}\subseteq X\setminus \bigcup_{A\in\mathcal A}\phi(A)$. Since $D$ is dense in $X$, we get that $D\cap F_\mathcal{A}\neq \emptyset$, which in turn implies that $D\setminus \bigcup_{A\in\mathcal A}\phi(A)\neq \emptyset$. Hence $D$ is not $\NC(X)$-compact in $X$.  

 In order to show that $D$ is not $\NC(Y)$-compact in $Y$ define the map $\psi:\NC(Y)\rightarrow o(Y)$ as follows: for each $A\in\NC(Y)$ let $\psi(A)=\phi(A\cap X)\cup (Y\setminus X)$. Let us check that the map $\psi$ is well defined. Since $Q$ is an open subset of $Y$, for each nowhere dense subset $A\subseteq Y$ the set $A\cap X$ is nowhere dense in $X$. Fix a subset $B\subseteq Y$ such that $\cl_Y(B)$ is compact. Since $X$ is a closed subset of $Y$, the set $\cl_Y(B)\cap X$ is compact, implying that $B\cap X\subseteq \cl_Y(B)\cap X\in \NC(X)$. Thus $A\cap X\in\NC(X)$ for each $A\in\NC(Y)$. It is easy to see that $A\subseteq \psi(A)$ for each $A\in\NC(Y)$. 
 Since $\phi(A\cap X)$ is open in $X$, there exists an open set $W$ in $Y$ such that $W\cap X=\phi(A\cap X)$. Observe that $$\psi(A)=\phi(A\cap X)\cup (Y\setminus X)=W\cup (Y\setminus X),$$ which implies that $\psi(A)$ is open for each $A\in\NC(Y)$. Hence the map $\psi$ is indeed well defined. 
 
  For each finite subset $\mathcal A\subseteq \NC(Y)$ we have $$D\setminus \bigcup_{A\in\mathcal A}\psi(A)=D\setminus \bigcup_{A\in\mathcal A}\phi(A\cap X) \neq \emptyset.$$ Hence $D$ is not $\NC(Y)$-compact in $Y$, as required.
\end{proof}

We are in a position to prove Theorem~\ref{appl}, which we restate for the reader's convenience. 

\Polishappl*

\begin{proof}
The group $G$ is non-compact, as $S$ has non-compact closure. Thus $\NC(G)$ is an ideal. Since the group $G$ is left topological, left shifts in $G$ are homeomorphisms. It follows that $\NC(G)$ is left-invariant. 

Since $G$ is Polish and non-compact, it contains a closed discrete subset $\{a_n: n\in\N\}$. Put $A_0=\{a_{2n}: n\in\N\}$ and $A_1=\{a_{2n+1}: n\in\N\}$. It is clear that $A_0$ and $A_1$ are disjoint closed and discrete. By the normality of $G$, there exist open sets $W_0, W_1$ such that $A_0\subseteq W_0$, $A_1\subseteq W_1$, and $\overline{W_0}\cap \overline{W_1}=\emptyset$. It is easy to see that neither $W_0$ nor $W_1$ belongs to $\NC(G)$. Seeking a contradiction, assume that there exists an ultrafilter $u$ on $G$ such that $\{G\setminus \overline{U}: U\in u\}\subseteq \NC(G)$. Then either $W_0\in u$ or $G\setminus W_0\in u$. In the former case, we have that $W_1\subseteq G\setminus \overline{W_0}\in \NC(G)$, which is impossible. In the latter case, we have $W_0= G\setminus (G\setminus W_0)=G\setminus \overline{G\setminus W_0}\in \NC(G)$, which is also impossible. Thus the ideal $\NC(G)$ is not near-maximal.

Since $G$ is locally compact separable metrizable, we get that $G$ is $\sigma$-compact. 
Since $G$ is locally compact and $\sigma$-compact, there exists an increasing family $\{V_n: n\in\N\}\subseteq \NC(G)\cap o(G)$ such that $\bigcup_{n\in\N}V_n=G$. 
The family $\{V_n:n\in\N\}$ witnesses that the ideal $\NC(G)$ is exhaustive. 
Proposition~\ref{card} yields $\chi(\NC(G))\geq \aleph_1$.

Let $Q$ be an open subset of $G$ such that $\cl_G(Q)$ is not compact and $S$ is dense in $Q$. Since $G$ is separable metrizable, the set $S$ contains a countable dense subset $D$. Let $S'$ be a subsemigroup of $S$ generated by the set $D$. It is clear that $|S'|=\aleph_0<\aleph_1\leq \chi(\NC(G))$. Note that the subspace $\cl_G(Q)$ of $G$ is separable metrizable.  Theorem~\ref{dow} implies that $\cl_G(Q)$ has a remote point. Since $S'$ is dense in $Q$, Lemma~\ref{dense} implies that $S'$ is not $\NC(G)$-compact in $G$.
Hence Theorem~\ref{main}, applied to $G$, $S'$, $\NC(G)$ and $C$, yields an open set $P\in\NC(G)^+$ such that $C\subseteq P$, and for each finite coloring $c'$ of $S'\setminus P$ there exists a family $\mathcal A$ of sequences in $S'$ which satisfies conditions (i) - (iv) with respect to $c'$. Moreover, if $\NC(G)$ is nicely exhaustive with respect to $S'$, then we can chose the family $\mathcal A$ to satisfy conditions (i) - (v) with respect to $c'$.  

Fix a finite coloring $c:S\setminus P\rightarrow m$. Then it induces a finite coloring $c': S'\setminus P\rightarrow m$. It is clear that the aforementioned family $\mathcal A$ corespondent to the coloring $c'$ satisfies conditions (i) - (iv) with respect to the coloring $c$. 
Assume additionally that $\NC(G)$ is nicely exhaustive with respect to $S$. Then there exists an increasing family $\{B_i:i\in\N\}\subseteq \NC(G)\cap o(G)$ such that $\bigcup_{i\in\N}B_i=G$ and $S\setminus B_i$ is a nonempty left ideal in $S$ for each $i\in\N$. Since $S'\notin \NC(G)$, we get that $S'\setminus B_i\neq \emptyset$ for all $i\in\N$. Moreover, as $S'\setminus B_i=S'\cap (S\setminus B_i)$, it is straightforward to check that $S'\setminus B_i$ is a left ideal in $S'$.  Hence the ideal $\NC(G)$ is nicely exhaustive with respect to  $S'$ as well. Thus we can choose the aforementioned family $\mathcal A$ correspondent to the coloring $c'$ to satisfy conditions (i) - (v). It is clear that the family $\mathcal A$ satisfies conditions (i) - (v) with respect to the coloring $c$ as well.
\end{proof}


\section{Applications of Theorem~\ref{main} to other ideals on topological groups}\label{4}

A subset $A$ of a topological space $X$ is called {\em meager} if there exist nowhere dense subsets $B_n$, $n\in\N$ of $X$ such that $A\subseteq \bigcup_{n\in\N}B_n$. A space $X$ is called {\em Baire} if $X$ is not meager.
For a Baire space $X$ let $\mathcal{M}(X)$ denote the ideal consisting of meager subsets of $X$ 

Observe that in Theorem~\ref{appl} we cannot substitute the ideal $\NC(G)$ by  the ideal $\mathcal M(G)$, provided that $G$ is not discrete. Indeed, let $D$ be a countable dense subset of $G$. Then $D\in\mathcal {M}(G)$ and for each open neighborhood $U$ of $D$ the set $G\setminus U$ is nowhere dense and thus $G\setminus U\in\mathcal {M}(G)$. Hence we cannot find a family $\mathcal A$ of sequences in $G\setminus U$ which satisfies condition (iv) from Theorem~\ref{appl}. The following lemma shows the core of the above arguments.

\begin{lemma}
   Let $X$ be a space that contains a dense meager subset and $J$ be an ideal on $X$ such that $\mathcal M(X)\subseteq J$. Then $X$ is $J$-compact. 
\end{lemma}

\begin{proof}
Let $D$ be a dense meager subset of $X$. Fix any map $\phi: J\rightarrow o(X)$ such that $A\subseteq \phi(A)$. Note that $\phi(D)$ is an open dense subset of $X$. It follows that $E=X\setminus \phi(D)\in \mathcal M(X)\subseteq J$. Thus $X\subseteq \phi(D)\cup \phi(E)$, which implies that $X$ is $J$-compact.     
\end{proof}

On the other hand, the following weaker counterpart of Corollary~\ref{NC} holds, where $\MC(\mathbb R^n)$ denotes the ideal on $\mathbb R^n$ consisting of sets that are contained in the union of a compact subset of $\mathbb R^n$ and a meager subset of $\mathbb R^n$.

\begin{proposition}
 Let $C\in\MC(\mathbb R^n)$. Then there exists an open set $P\in\MC(\mathbb R^n)^+$ such that $C\subseteq P$, and for each finite coloring of $\mathbb R^n\setminus P$ there exists an infinite sequence $A$ in $\mathbb R^n\setminus P$ such that $\FS(A)$ is a closed discrete monochromatic set.   
\end{proposition}

\begin{proof}
There exist a compact subset $K\subseteq \mathbb R^n$ and a meager subset $D\subseteq \mathbb R^n$ such that $C\subseteq K\cup D$. Since $K$ is compact, there exists $k\in\N$ such that $$K\subseteq B_k=\{(x_1,\ldots,x_n)\in\mathbb R^n: |x_i|<k \hbox{ for all }i\leq n\}.$$ Put $C'=B_k\cup D$.
It is easy to see that for each $n\in \N$ we have $$\frac{1}{n}C'=\{(x_1,\ldots, x_n)\in \mathbb R^n: (nx_1,\ldots, nx_n)\in C'\}\subseteq B_k\cup \frac{1}{n}D,$$
and the set $\frac{1}{n}D$ is meager.
For an element $x=(x_1,\ldots, x_n)$ of $\mathbb R^n$ let $$S(x)=\{(nx_1,\ldots, nx_n): n\in\N\}.$$  To derive a contradiction, assume that $C'\cap S(x)\neq \emptyset$ for each $x\in \mathbb R^n$. Then 
$$\mathbb R^n=\bigcup_{n\in\N}\frac{1}{n}C'\subseteq B_k\cup \bigcup_{n\in\N}\frac{1}{n}D.$$ The above inclusion contradicts the Baire Category Theorem applied to $\mathbb R^n\setminus B_k$, as the set $\bigcup_{n\in\N}\frac{1}{n}D$ is meager. Hence there exists $x\in \mathbb R^n$ such that $S(x)\cap C'=\emptyset$. It is routine to check that $S(x)$ is an infinite closed discrete subsemigroup of $\mathbb R^n$. Put $P=\mathbb R^n\setminus S(x)$. It is clear that $C\subseteq C'\subseteq P$ and $P\in\MC(\mathbb R^n)^+$. By Theorem~\ref{Hindman}, for each finite coloring of $S(x)=\mathbb R^n\setminus P$ there exists a sequence $A=\{n_ix: i\in\N\}$ in $S(x)$ such that $\FS(A)$ is monochromatic.
It is straightforward to check that $\FS(A)$ is closed and discrete in $\mathbb R^n$. 
\end{proof}

Recall that each Hausdorff locally compact topological group $G$ admits a {\em Haar} measure, i.e. a $\sigma$-additive outer regular Borel measure $\mu$, which is invariant under left shifts and takes finite values on compact sets. Moreover, if $\mu_1$, $\mu_2$ are two Haar measures on $G$, then there exists a positive real $a$ such that $\mu_1(A)=a\mu_2(A)$ for each Borel subset $A\subseteq G$.  For more on Haar measures see~\cite[Section 2.2]{Fol}. For a fixed Haar measure $\mu$ on a locally compact Hausdorff topological group $G$, the {\em outer Haar measure} $\mu^*$ is defined for each subset $A$ of $G$ as follows: 
$$\mu^*(A)=\inf\{\mu(U): U \hbox{ is open and } A\subseteq U\}.$$  Further for a given Hausdorff locally compact topological group $G$, the symbol $\mu$ denotes a Haar measure on $G$. Since the Haar measure on such groups is unique up to a scaling constant, the following definition is independent of the specific choice of a Haar measure.

\begin{definition}
For a non-compact locally compact $T_1$ topological group $G$ let $\mathcal{H}_{\text{fin}}(G) = \{A \subseteq G : \mu^*(A) < \infty\}$.
\end{definition}

\begin{lemma}\label{primeopen}
Let $G$ be a locally compact non-compact $T_1$ topological group. Then $\Hf(G)$ is an ideal on $G$ that is not a near-maximal ideal.
\end{lemma}

\begin{proof}
Let us first show that $\mu(G)=\infty$. 
By~\cite[Proposition 2.19]{Fol}, for each open subset $U$ of $G$ we have $\mu(U)>0$. Fix an open neighborhood $V$ of $1_G$ with compact closure. Since $G$ is a topological group, the set $VV^{-1}$ also has compact closure. Put $x_1=1_G$ and fix any $x_2\in G\setminus VV^{-1}$, which exists as $G$ is not compact. Assume that we already constructed points $\{x_1,\ldots,x_n\}\subseteq G$ such that $x_i\in G\setminus \bigcup_{j<i}x_jVV^{-1}$. Let $x_{n+1}$ be any point of $G\setminus  \bigcup_{j\leq n}x_jVV^{-1}$, which exists as $\bigcup_{j\leq n}x_jVV^{-1}$ is compact and $G$ is not. This way we obtain a sequence $\{x_n:n\in\N\}$ in $G$. It is easy to check that $x_iV\cap x_jV=\emptyset$ whenever $i\neq j$. Since $\mu$ is $\sigma$-additive and invariant under left shifts, we get $$\mu(G)\geq \mu(\bigcup_{n\in\N}x_nV)=\sum_{n\in \N}\mu(V)=\infty.$$ 

At this point it is easy to see that $\Hf(G)$ is an ideal.
To derive a contradiction, assume that there exists an ultrafilter $u$ on $G$ such that $\{G\setminus\overline{U}: U\in u\}\subseteq \Hf(G)$. 
Note that either $G\setminus(\bigcup_{n\in \N}x_{2n} V)\in u$, or $G\setminus(\bigcup_{n\in \N}x_{2n+1} V)\in u$. Each case provides $U\in u$ such that $\mu(G\setminus \overline{U})=\infty$. The obtained contradiction yields that $\Hf(G)$ is not a near-maximal ideal.
\end{proof}

For a given non-compact locally compact Polish topological group $G$, it is straightforward to check that $\Hf(G)$ is a left-invariant exhaustive ideal which has a base consisting of open sets. It follows that $\chi(\Hf(G))=|G|^+$ and each subset $A\notin \Hf(G)$ of $G$ is not $\Hf(G)$-compact in $G$. Then, taking into account Lemma~\ref{primeopen}, Theorem~\ref{main} implies the following:

\begin{theorem}\label{Hf}
   Let $G$ be a locally compact non-compact Polish topological group, $C$ be a fixed element of $\Hf(G)$, and $S$ be a subsemigroup of $G$ of infinite outer Haar measure. Then there exists an open set $P$ of infinite outer Haar measure such that $C\subseteq P$, and for every finite coloring of $S\setminus P$ there exists a family $\mathcal A$ of sequences in $S\setminus P$ which satisfies the following conditions:
   \begin{enumerate}[\rm(i)]
  \item each $A\in\mathcal A$ is closed and discrete in $G$;
  \item the set $\bigcup_{A\in\mathcal A}\FS(A)$ is monochromatic;
  \item $\FS(A)\cap \FS(B)=\emptyset$ for all distinct $A,B\in\mathcal A$;
  \item the set $\bigcup_{A\in\mathcal A}\FS(A)$ has infinite outer Haar measure.
    \end{enumerate}
  Moreover, if $\Hf(G)$ is additionally nicely exhaustive with respect to $S$, then 
  \begin{enumerate}[\rm(v)]
  \item $\FS(A)$ is a closed discrete subset of $G$ for all $A\in\mathcal A$.
  \end{enumerate}
\end{theorem}

Recall that the usual Lebesgue measure on $\mathbb R^n$ is in fact a Haar measure. Further, if we are dealing with $\mathbb R^n$, then $\mu$ denotes the Lebesgue measure. For each $m\in\N$ let 
\begin{equation}\label{Bm}
 B_m=\{(x_1,\ldots, x_n)\in\mathbb R^n: |x_i|<m \hbox{ for all }i\leq n\}.   
\end{equation}

The following simple lemma is well-known. Nevertheless, we include its short proof.

\begin{lemma}\label{Rn}
Let $S$ be a subsemigroup of $\mathbb R^n$ such that $\mu^*(S)>0$. Then $\mu^*(S)=\infty$.    
\end{lemma}

\begin{proof}
It is easy to see that there exists $m\in\N$ such that $\mu^*(S\cap B_m)>0$. Fix a non-zero $x\in S$. There exists a sequence $\{k_i: i\in\N\}\subseteq \N$ such that for each $i\neq j$ we have
$$(k_ix+B_m)\cap (k_jx+B_m)=\emptyset.$$ Since $x\in S$, we have $nx\in S$ for all $n\in\N$. It follows $k_ix+(S\cap B_m)\subseteq S$ for all $i\in\N$. By the choice of $k_i$, we have $(k_ix+(S\cap B_m))\cap (k_jx+(S\cap B_m))=\emptyset$ for all distinct $i,j\in\N$. Since the outer Lebesgue measure is shift-invariant, $\mu^*(k_ix+(S\cap B_m))=\mu^*(S\cap B_m)$ for all $i\in\N$. Thus $\mu^*(S)\geq \sum_{i\in\N}\mu^*(S\cap B_m)=\infty$.   
\end{proof}

By $\nN$ we denote the least cardinality of a subset of $\mathbb R^n$ with positive outer Lebesgue measure. It is known that $\aleph_1\leq \nN\leq 2^{\aleph_0}$, see~\cite{Blass}.
The family $\{B_m:m\in\N\}$ (see Equation~\ref{Bm} above) witnesses that the ideal $\Hf(\mathbb R^n)$ is nicely exhaustive with respect to any subsemigroup $S\subseteq \mathbb R^n_+$ of infinite outer Lebesgue measure. Thus Theorem~\ref{Hf} and Lemma~\ref{Rn} imply the following:

\begin{corollary}
   Let $C\in \Hf(\mathbb R^n)$, and $S$ be a subsemigroup of $\mathbb R^n_+$ with $\mu^*(S)>0$. Then there exists an open set $P\supseteq C$ of infinite outer Haar measure such that for every finite coloring of $S\setminus P$ there exists a family $\mathcal A$ of cardinality $\geq \nN$ consisting of sequences in $S\setminus P$ which satisfies the following conditions:
   \begin{enumerate}[\rm(i)]
  \item $\FS(A)$ is a closed discrete subset of $\mathbb R^n$ for all $A\in\mathcal A$;
  \item the set $\bigcup_{A\in\mathcal A}\FS(A)$ is monochromatic;
  \item $\FS(A)\cap \FS(B)=\emptyset$ for all distinct $A,B\in\mathcal A$;
  \item the set $\bigcup_{A\in\mathcal A}\FS(A)$ has infinite outer Lebesgue measure.
    \end{enumerate}
\end{corollary}

The inequality $|\mathcal A|\geq \nN$ in the above corollary is forced by condition~(iv).

Since the Haar measure on locally compact Hausdorff topological groups is unique up to a scaling constant, the following two definitions are independent of the specific choice of a Haar measure.

\begin{definition}\label{test}
Let $G$ be a topological group. A family $\{B_n: n\in\N\}\subseteq o(G)$ is called a {\em F{\o}lner chain} if for each $n\in\N$ it satisfies the following conditions:
\begin{enumerate}
    \item $\overline{B_n}\subseteq B_{n+1}$;
    \item $\overline{B_n}$ is compact;
    \item $\bigcup_{n\in\N}B_n=G$;
    \item $\mu(\overline{B_n}\setminus B_n)=0$;
    \item for each $g\in G$ we have 
    $$\lim_{n\to \infty}\frac{\mu(gB_n\setminus B_n)}{\mu(B_n)}=0.$$
\end{enumerate}
\end{definition}

Note that every topological group which possesses a F{\o}lner chain is locally compact, non-compact and $\sigma$-compact. 

\begin{definition}
 Let $\Delta=\{B_n:n\in\N\}$ be a F{\o}lner chain in a $T_1$ topological group $G$. A subset $A\subseteq G$ is said to have {\em asymptotic density zero with respect to $\Delta$} if there exists a Borel subset $A'\supseteq A$ such that  $$\lim_{n\to\infty} \frac{\mu(A'\cap B_n)}{\mu(B_n)}=0.$$ 
 Let $\Z_{\Delta}(G)$ denote the ideal on $G$ consisting of the sets with asymptotic density zero with respect to the F{\o}lner chain $\Delta$.
\end{definition}

\begin{lemma}\label{asy}
 Let $\Delta=\{B_n:n\in\N\}$ be a F{\o}lner chain in a $T_1$ topological group $G$. Then $\Z_{\Delta}(G)$ is not a near-maximal ideal.
\end{lemma}

\begin{proof}
Suppose towards a contradiction that there exists an ultrafilter $u$ on $G$ such that $\{G \setminus \overline{U}: U \in u\}\subseteq \mathcal{Z}_{\Delta}(G)$. 
 Since $G$ is locally compact and non-compact, the proof of Lemma~\ref{primeopen} yields $\mu(G)=\infty$. It follows that $\mu(B_n) \to \infty$. Then there exists a strictly increasing sequence of indices $\{n_k: k\in\N\} \subseteq \mathbb{N}$ such that:
$$\frac{\mu(\overline{B_{n_k}})}{\mu(B_{n_{k+1}})} < \frac{1}{2^k} \quad \text{for all } k \ge 1.$$
Let
$$W = \bigcup_{k\in \N}\left(B_{n_{4k}} \setminus \overline{B_{n_{4k-1}}}\right), \quad \hbox{ and } \quad W' = \bigcup_{k\in\N} \left(B_{n_{4k+2}} \setminus \overline{B_{n_{4k+1}}}\right).$$

For each $k \ge 1$, we have $\overline{B_{n_{4k}} \setminus \overline{B_{n_{4k-1}}}} \subseteq \overline{B_{n_{4k}}} \subseteq B_{n_{4k+1}}$. Note that the family $\{B_{n_{4k}} \setminus \overline{B_{n_{4k-1}}}: k\in\N\}$ is locally finite.  Hence $\overline{W} \subseteq \bigcup_{k=1}^\infty B_{n_{4k+1}}$. By the definition of $W'$, we have $\overline{W} \cap W' = \emptyset$.

Since $u$ is an ultrafilter, either $G \setminus W \in u$ or $W \in u$. First assume that $G \setminus W \in u$. By the choice of $u$, $W=G \setminus \overline{G\setminus W}  \in \mathcal{Z}_{\Delta}(G)$. However, $$\frac{\mu(W \cap B_{n_{4k}})}{\mu(B_{n_{4k}})} \geq \frac{\mu(B_{n_{4k}} \setminus \overline{B_{n_{4k-1}}})}{\mu(B_{n_{4k}})} = 1 - \frac{\mu(\overline{B_{n_{4k-1}}})}{\mu(B_{n_{4k}})} > 1 - \frac{1}{2^{4k-1}}.$$
It follows that $1$ is an accumulation point of the sequence $\{\frac{\mu(W\cap B_n)}{\mu(B_n)}: n\in\N\}$, contradicting $W \in \mathcal{Z}_{\Delta}(G)$. Assume that $W \in u$. Then $W'\subseteq G \setminus \overline{W} \in \mathcal{Z}_{\Delta}(G)$. Observe that 
$$ \frac{\mu(W' \cap B_{n_{4k+2}})}{\mu(B_{n_{4k+2}})} \ge\frac{\mu(B_{n_{4k+2}} \setminus \overline{B_{n_{4k+1}}})}{\mu(B_{n_{4k+2}})}\geq 1 - \frac{\mu(\overline{B_{n_{4k+1}}})}{\mu(B_{n_{4k+2}})} > 1 - \frac{1}{2^{4k+1}}.$$
It follows that $1$ is an accumulation point of the sequence $\{\frac{\mu(W'\cap B_n)}{\mu(B_n)}: n\in\N\}$, contradicting $W' \in \mathcal{Z}_{\Delta}(G)$. 
These contradictions imply that $\mathcal{Z}_{\Delta}(G)$ is not a near-maximal ideal.
\end{proof}
%


\begin{theorem}\label{AD}
Let $\Delta$ be a F{\o}lner chain in a $T_1$ topological group $G$, $C\in \Z_{\Delta}(G)$, and $S$ be a subsemigroup of $G$ such that $S\notin\Z_{\Delta}(G)$. Then there exists an open set $P\in \Z_{\Delta}(G)^+$ such that $C\subseteq P$, and for every finite coloring of $S\setminus P$ there exists a family $\mathcal A$ of sequences in $S\setminus P$ which satisfies the following conditions: 
   \begin{enumerate}[\rm(i)]
   \item  each $A\in\mathcal A$ is closed and discrete in $G$;
  \item the set $\bigcup_{A\in\mathcal A}\FS(A)$ is monochromatic;
  \item $\FS(A)\cap \FS(B)=\emptyset$ for all distinct $A,B\in\mathcal A$;
   \item $\bigcup_{A\in\mathcal A}\FS(A)\in\Z_{\Delta}(G)^+$.
  \end{enumerate}
   Moreover, if $\Z_{\Delta}(G)$ is additionally nicely exhaustive with respect to $S$, then
  \begin{enumerate}[\rm(v)]
  \item $\FS(A)$ is a closed discrete subset of $G$ for all $A\in\mathcal A$.
  \end{enumerate}  
\end{theorem}

\begin{proof}
To apply Theorem~\ref{main}, we need to check that $\Z_{\Delta}(G)$ is a left-invariant exhaustive non-near-maximal ideal, $S$ is not $\Z_{\Delta}(G)$-compact in $G$, and $|S|<\chi(\Z_{\Delta}(G))$.

By Lemma~\ref{asy}, $\Z_{\Delta}(G)$ is not a near-maximal ideal.

Let us check that the ideal $\Z_{\Delta}(G)$ is left-invariant. Fix any $A\in \Z_{\Delta}(G)$ and $g\in G$.  
Since the Haar measure $\mu$ is invariant under left shifts, and $A\cap g^{-1}B_n\subseteq (A\cap B_n)\cup (g^{-1}B_n\setminus B_n)$ we get the following:
$$\frac{\mu(gA \cap B_n)}{\mu(B_n)} = \frac{\mu(g^{-1}(gA \cap B_n))}{\mu(B_n)}= \frac{\mu(A \cap g^{-1}B_n)}{\mu(B_n)} \leq \frac{\mu(A\cap B_n)}{\mu(B_n)} + \frac{\mu(g^{-1}B_n \setminus B_n)}{\mu(B_n)}.$$
Taking into account that $A\in\Z_{\Delta}(G)$ and Condition (5) in Definition~\ref{test}, we obtain the following:  
$$\lim_{n\to\infty}\frac{\mu(gA \cap B_n)}{\mu(B_n)} \leq \lim_{n\to\infty}\frac{\mu(A \cap B_n)}{\mu(B_n)} + \lim_{n\to\infty}\frac{\mu(g^{-1}B_n \setminus B_n)}{\mu(B_n)}=0+0=0.$$
Hence the ideal $\Z_{\Delta}(G)$ is left-invariant.

By the proof of Lemma~\ref{primeopen}, $\mu(G)=\infty$. 
The compactness of $\overline{B_n}$ implies that $\mu(B_n)<\infty$ for all $n\in\N$. Thus $\Delta\subseteq \Z_{\Delta}(G)$ and, consequently, the ideal $\Z_{\Delta}(G)$ is exhaustive.

We claim the ideal $\Z_{\Delta}(G)$ has a base consisting of open sets. Indeed, fix any $A\in \Z_{\Delta}(G)$. Without loss of generality we can assume that $A$ is Borel. For each $m\in\N$ let $A_m=A\cap (B_m\setminus \overline{B_{m-1}})$, where we agree that $B_0=\emptyset$. Recall that the Haar measure $\mu$ is outer regular, i.e. for each Borel set $L$ we have $$\mu(L)=\inf\{\mu(U): L\subseteq U\hbox{ and }U\hbox{ is open}\}.$$ Since $\mu(\overline{B_m}\setminus B_m)=0$ (see condition (4) of Definition~\ref{test}), the outer regularity of $\mu$ implies that for each $m\in\N$ there exists an open subset $V_m\subseteq G$ such that $$\overline{B_m}\setminus B_m\subseteq V_m\qquad \hbox{ and }\qquad \mu(V_m)\leq \frac{1}{2^{m}}.$$
 Since the sets $A_m$, $m\in\N$ are Borel, for each $m\in\N$ the outer regularity of $\mu$ yields an open subset $U_m\subseteq G$ such that 

$$A_m\subseteq U_m, \qquad U_{m}\cap \overline{B_{m-1}}=\emptyset \qquad \hbox{and} \qquad \mu(U_m)\leq \mu(A_m)+\frac{1}{2^m}.$$ 

Put $U=\bigcup_{m\in \N}U_m$, $V=\bigcup_{m\in\N}V_m$ and $W_A=U\cup V$. Observe that for each $m\in\N$ we have 
$W_A\cap B_m\subseteq V\cup\bigcup_{i\leq m}U_i$. 
It follows that
$$
\mu(W_A\cap B_m)\leq 1+\sum_{i=1}^{m}\mu(U_i)\leq 1+\sum_{i=1}^{m}(\mu(A_i)+\frac{1}{2^i})\leq \mu(A\cap B_{m})+2.
$$
Hence, taking into account that $\lim_{n\to \infty}\mu(B_n)=\infty$, we get the following: 
$$
\lim_{m\to\infty}\frac{\mu(W_A\cap B_m)}{\mu(B_m)}\leq \lim_{m\to\infty}\frac{\mu(A\cap B_m)}{\mu(B_m)}+\lim_{m\to\infty}\frac{2}{\mu(B_m)}=0+0=0.
$$
Thus $W_A\in \Z_{\Delta}(G)$, and consequently the ideal $\Z_{\Delta}(G)$ has a base consisting of open sets. It follows that $\chi(\Z_{\Delta}(G))=|G|^+>|S|$.

Since $S\notin \Z_{\Delta}(G)$, the map $\phi: \Z_{\Delta}(G)\rightarrow o(G)$ defined by $\phi(A)=W_A$ witnesses that $S$ is not $\Z_{\Delta}(G)$-compact in $G$. 
\end{proof}


Note that $\Delta=\{B_m:m\in\N\}$ (see Equation~\ref{Bm}) is a F{\o}lner chain in $\mathbb R^n$. It is easy to check that the family $\Delta\subseteq \Z_{\Delta}(\mathbb R^n)$ witnesses that the ideal $\Z_{\Delta}(\mathbb R^n)$ is nicely exhaustive with respect to every subsemigroup $S$ of $\mathbb R^n_+$ such that $S\notin \mathcal Z_{\Delta}(\mathbb R^n)$. Then 
Theorem~\ref{AD} applied to $\mathbb R^n$ and the aforementioned F{\o}lner chain $\Delta$ implies the following.

\begin{corollary}
Let $S$ be a subsemigroup of $\mathbb R^n_+$ such that $S\notin\Z_\Delta(\mathbb R^n)$, and $C\in \Z_{\Delta}(\mathbb R^n)$. Then there exists an open set $P\in \Z_\Delta(\mathbb R^n)^+$ such that $C\subseteq P$, and for every finite coloring of $S\setminus P$ there exists a family $\mathcal A$ of cardinality $\geq \nN$ consisting of sequences in $S\setminus P$ which satisfies the following conditions: 
   \begin{enumerate}[\rm(i)]
   \item $\FS(A)$ is a closed discrete subset of $\mathbb R^n$ for all $A\in\mathcal A$;
  \item the set $\bigcup_{A\in\mathcal A}\FS(A)$ is monochromatic;
  \item $\FS(A)\cap \FS(B)=\emptyset$ for all distinct $A,B\in\mathcal A$;   
   \item $\bigcup_{A\in\mathcal A}\FS(A)\in \Z_\Delta(\mathbb R^n)^+$.
  \end{enumerate}  
\end{corollary}

The inequality $|\mathcal A|\geq \nN$ in the above corollary follows from condition (iv) and the fact that each element of $\Z_\Delta(\mathbb R^n)^+$ has infinite outer Lebesgue measure.

A topological group $G$ is called {\em precompact} if for each open neighborhood $U$ of $1_G$ there exists a finite set $\{g_1,\ldots,g_n\}\subseteq G$ such that $\bigcup_{i\leq n}g_iU=G$.
By~\cite[Theorem 3.7.10]{AT} a topological group $G$ is precompact if and only if $G$ is topologically isomorphic to a subgroup of a compact topological group. The following result implies that Theorems~\ref{main},~\ref{appl},~\ref{Hf} and~\ref{AD} are no longer valid for precompact topological groups.

\begin{proposition}
Let $G$ be a precompact topological group. Then $1_G\in \overline{\FS(A)}$ for each infinite sequence $A\subseteq G$.     
\end{proposition}

\begin{proof}
Let $H$ be a compact topological group which contains $G$ as a topological subgroup. Fix an infinite sequence $A\subseteq G$. By~\cite[Theorem 5.12]{HS}, there exists an idempotent ultrafilter $u$ on $H$ such that $\FS(A)\in u$. It follows from~\cite[Proposition~5.1]{BZl} that $u$ converges to $1_H=1_G$. Thus $1_G\in \overline{\FS(A)}$, as required.   
\end{proof}

\end{document}